\documentclass[11pt,reqno]{article}

\usepackage{pb-diagram}

\usepackage{pb-diagram}

\usepackage[ruled]{algorithm2e}

\usepackage{amsfonts}
\usepackage{amscd}
\usepackage{color}
\usepackage{amsmath}
\usepackage{float}
\usepackage{amsthm}
\usepackage{amssymb}
\usepackage{mathrsfs}
\usepackage{amstext}
\usepackage{graphicx}
\usepackage{tikz}
\usepackage{multirow}
\usepackage{verbatim}
\newcommand{\reallywidehat}[1]{%
	\savestack{\tmpbox}{\stretchto{%
			\scaleto{%
				\scalerel*[\widthof{\ensuremath{#1}}]{\kern-.6pt\bigwedge\kern-.6pt}%
				{\rule[-\textheight/2]{1ex}{\textheight}}
			}{\textheight}%
		}{0.5ex}}%
	\stackon[1pt]{#1}{\tmpbox}%
}

\numberwithin{equation}{section}
\theoremstyle{plain}
\newtheorem{satz}{Theorem}[section]
\newtheorem{defi}[satz]{Definition}
\newtheorem{cor}[satz]{Corollary}
\newtheorem{lem}[satz]{Lemma}
\newtheorem{prop}[satz]{Proposition}

\newcommand{\re}{\ensuremath{\mathbb{R}}}\newcommand{\N}{\ensuremath{\mathbb{N}}}
\newcommand{\zz}{\ensuremath{\mathbb{Z}}}

\newcommand{\Z}{{\ensuremath{\zz}^d}}

\newcommand{\R}{\ensuremath{{\re}^d}}
\newcommand{\qe}{\ensuremath{\mathbb{Q}}}

\newcommand{\Bspt}{\ensuremath{\mathbf{B}^r_{p,\theta}}}

\newcommand{\Bo}{\ensuremath{\mathring{\mathbf B}_{p,\theta}^r}}

\newcommand{\cf}{\ensuremath{\mathcal F}}

\newcommand{\supp}{{\rm supp \, }}

\newcommand{\bproof}{\begin{proof}}
\newcommand{\eproof}{\end{proof}}

\newlength{\fixboxwidth}
\usepackage[normalem]{ulem}

\newcommand{\be}{\begin{equation}}
\newcommand{\ee}{\end{equation}}
\newcommand{\beq}{\begin{eqnarray}}
\newcommand{\beqq}{\begin{eqnarray*}}
\newcommand{\eeq}{\end{eqnarray}}
\newcommand{\eeqq}{\end{eqnarray*}}

\begin{document}
\title{On the orthogonality of the Chebyshev-Frolov lattice and applications}

\author{Christopher Kacwin\footnote{Email: kacwin@ins.uni-bonn.de}, Jens
Oettershagen\footnote{Email: oettershagen@ins.uni-bonn.de}, Tino Ullrich\footnote{Corresponding author. Email:
tino.ullrich@hausdorff-center.uni-bonn.de} \\\\
	Institute for Numerical Simulation, 53115 Bonn, Germany\\Hausdorff-Center for Mathematics\\
}

\date{\today}

\maketitle


\begin{abstract} 
We deal with lattices that are generated by the Vandermonde matrices associated to the roots of Chebyshev-polynomials. 
If the dimension $d$ of the lattice is a power of two, i.e. $d=2^m, m \in \mathbb{N}$, the resulting 
lattice is an admissible lattice in the sense of Skriganov \cite{Sk94}. These are related to the Frolov cubature formulas, 
which recently drew attention due to their optimal convergence rates \cite{UU15} in a broad range of Besov-Lizorkin-Triebel spaces. 
We prove that the resulting lattices are orthogonal and possess a lattice representation matrix with entries not larger than $2$ (in modulus). 
This allows for an efficient enumeration of the Frolov cubature nodes in the $d$-cube $[-1/2,1/2]^d$ up to dimension $d=16$.
\end{abstract}

\section{Introduction}

A lattice is a set of points in $\R$ given by 
$$
  \Gamma_A = A(\Z) = \Big\{\sum_{j=1}^d k_j a_j : (k_1,\ldots,k_d) \in \Z \Big\}\,,
$$
where the $a_j \in \R$ are the columns of the generating matrix $A \in \re^{d\times d}$. Of particular interest are \emph{admissible lattices} $\Gamma$ in the sense of Skriganov \cite{Sk94} which fulfill
\begin{equation}\label{adm}
    \inf\limits_{\gamma \in \Gamma \setminus \{0\}} \Big|\prod\limits_{i=1}^d \gamma_i \Big| > 0\,.
\end{equation}
This immediately implies that any vector in the lattice (except the zero vector) consists of only non-vanishing components. 
However, the condition in \eqref{adm} is much stronger than that and crucial for the performance of the Frolov \cite{Fr76} cubature
formula for multivariate functions with $\supp f \subset [-1/2,1/2]^d$ given by 
\begin{equation}\label{Fr}
  \Phi(n,A;f) := \frac{1}{n} \sum\limits_{k \in \Z} f\left( ( n\det(A) )^{-1/d} A k \right) \, ,\quad n\in \N\,,
\end{equation}
see also Bykovskii \cite{By85}, Dubinin \cite{Du92, Du97}, Temlyakov
\cite{Tem93, Tem03} and the recent papers by M. Ullrich \cite{MU14, MU16}, Nguyen, M. Ullrich and T. Ullrich \cite{UU15, NUU15}\footnote{The modifications proposed in \cite{NUU15} lead to optimal cubature formulae also for functions without homogeneous boundary condition.} 
and Krieg, Novak \cite{KrNo16}. Its asymptotic performance is well-understood as it provides optimal
convergence rates for several classes of functions with bounded mixed derivative and compact support, given that $\Gamma_A = A(\Z)$ is admissible.

However, there is a degree of freedom in choosing the lattice generating matrix $A$ in \eqref{Fr} such that property
\eqref{adm} holds which significantly affects the numerical properties of the algorithm. In the original paper by
Frolov \cite{Fr76} a Vandermonde matrix
\begin{equation}\label{vanderm}
    A = \left(\begin{array}{cccc}
                1 & \xi_1 & \cdots & \xi_1^{d-1}\\
                1 & \xi_2 & \cdots & \xi_2^{d-1}\\
                \vdots & \vdots & \ddots & \vdots\\
                1 & \xi_{d} & \cdots & \xi_d^{d-1}
              \end{array}\right)\,
\end{equation}
has been considered, where $\xi_1,\ldots,\xi_d$ are the real roots of an irreducible polynomial over $\mathbb{Q}$,
e.g., $P_d(x):=\prod_{j=1}^d (x-2j+1)-1$\,. The general principle of this construction has been elaborated in detail by Temlyakov
in his book \cite[IV.4]{Tem93} based on results on algebraic number theory, see Borevich, Shafarevich \cite{BoSh66} or Gruber, Lekkerkerker \cite{GrLekk87}. 

The above polynomial $P_d$ has a striking disadvantage, namely that the real roots of the polynomials grow with $d$ and
therefore the entries in $A$ get huge due to the Vandermonde structure. In fact, sticking to the structure
\eqref{vanderm}, it seems to be a crucial task to find proper irreducible polynomials with real roots of small modulus. In
\cite[IV.4]{Tem93} Temlyakov proposed the use of rescaled Chebyshev polynomials $Q_d$. To be more precise we use
for $x\in [-2,2]$
\begin{equation}\label{Cheb}
Q_d(x) = 2 T_d(x/2)\quad\mbox{with}\quad T_d(\cdot):=\cos(d\arccos(\cdot))\,.
\end{equation}
The polynomials $Q_d$ belong to $\mathbb{\zz}[x]$ and have leading coefficient $1$. Its roots are real and given by
\begin{equation}\label{zeros}
    \xi_{k} = 2\cos\Big(\frac{\pi(2k-1)}{2d}\Big) \quad,\quad k=1,...,d\,.
\end{equation}
In the sequel we will denote the Vandermonde matrix \eqref{vanderm} with the scaled Chebyshev roots \eqref{zeros} by the letter $T$ 
and call the corresponding lattice $\Gamma_T = T(\Z)$ a {\em Chebyshev lattice}.
Our main result reads as follows. 

\begin{satz}\label{main} The $d$-dimensional Chebyshev lattice $\Gamma_T = T(\Z)$ is orthogonal. In particular, there
exists a lattice representation $\tilde{T}=TS$ with $S\in \text{SL}_d(\zz)$ such that 
\begin{itemize}
 \item[(i)] $\tilde{T}_{k,\ell} \in [-2,2]$ for $k,\ell  =1,...,d$ and
 \item[(ii)] $\tilde{T}^{\top}\tilde{T} = \mathrm{diag}(d, 2d,\ldots,2d)$.
 \end{itemize}
\end{satz}

However, Chebyshev-polynomials are not always irreducible over $\mathbb{Q}$. In fact, the polynomials $Q_d$ are irreducible if and only if $d = 2^m$ \cite[IV.4]{Tem93}. Hence, 
a Chebyshev lattice $\Gamma_T$ is admissible if and only if $d=2^m$. In that case we call $\Gamma_T$ a {\em Chebyshev-Frolov lattice} and obtain the following corollary. 
\begin{cor} If $d=2^m$ for some $m\in \N$ the Chebyshev-Frolov lattice $\Gamma_T = T(\Z)$ and its dual lattice are both 
orthogonal and admissible. In particular, there is a lattice representation
for $\Gamma$ given by $\tilde{T} = Q D$ with a diagonal matrix $D = \mathrm{diag}(\sqrt{d},\sqrt{2d},...,\sqrt{2d})$  and an orthogonal matrix $Q$. For the dual
lattice $\Gamma^{\perp}$ we have the representation $\tilde{T}^{\perp} =Q D^{-1}$. 
\end{cor}
This observation significantly affects the runtime of an algorithm  enumerating the lattice points belonging to $[-1/2,1/2]^d$
which represents a first non-trivial step in the implementation of the Frolov cubature formula, see Section 4, 5. By heavily relying on the orthogonality of the 
respective Chebyshev lattice 
we give an upper bound in Section 5 for the number of points which have to be seen in order to enumerate the $N$ lattice points in the $d$-cube $[-1/2,1/2]^d$. We confirm the result with 
some numerical tests up to dimension $d=16$. It turns out that we do not have to touch more than $2.07^d\cdot N$ points of the lattice. 

Let us finally refer to a forthcoming paper by M. Ullrich and the authors for the implementation and comparison of the performance of 
Frolov's method to other up to date cubature formulas. 
There we will also pay special attention to the case $d \neq 2^m$.

{\bf Notation.} As usual $\N$ denotes the natural numbers, 
$\zz$ denotes the integers, 
and $\re$ the real numbers
. 
The letter $d$ is always reserved for the underlying dimension in $\R, \Z$ etc. We denote
with $(x,y)$ 
the usual Euclidean inner product in $\R$
. For $0<p\leq \infty$ we denote with $|\cdot |_p$ and $\|\cdot \|_p$ the ($d$-dimensional) discrete $\ell_p$-norm and the continuous $L_p$-norm on $\R$, respectively, 
where $B_p^d$ denotes the respective unit ball in $\re^d$.
With $\cf$ we denote the Fourier transform given by $\cf f(\xi):=(2\pi)^{-d/2}\int_{\R} f(x)\exp(-ix\cdot \xi)\,\mathrm{d}x$ for 
a function $f\in L_1(\R)$ and $\xi \in \R$. For two sequences of real numbers $a_n$ and $b_n$ we will write 
$a_n \lesssim b_n$ if there exists a constant $c>0$ such that 
$a_n \leq c\,b_n$ for all $n$. We will write $a_n \asymp b_n$ if 
$a_n \lesssim b_n$ and $b_n \lesssim a_n$. With $\mathrm{GL}_d:=\mathrm{GL}_d(\mathbb{R})$ we
denote the group of invertible matrices over $\mathbb{R}$, wheras $\mathrm{SO}_d:=\mathrm{SO}_d(\re)$ 
denotes the group of orthogonal matrices over $\re$ with unit determinant. With $\mathrm{SL}_d(\mathbb{Z})$ we 
denote the group of invertible matrices over $\zz$ with unit determinant.
The notation $D:=\mbox{diag}(x_1,...,x_d)$ with $x = (x_1,...,x_d) \in \R$ refers to the diagonal matrix $D \in \re^{d\times d}$ with $x$ at the diagonal. 
And finally, by $\mathbb{Z}[x]$ we denote the ring of polynomials with integer coefficients.

\section{Construction of admissible lattices}
In this section we will briefly recall the precise notions of a lattice, its dual lattice, orthogonal and admissible lattices. We will furthermore 
comment on different lattice representations. 

\begin{figure}[t]
\centering
	\includegraphics[height=6cm]{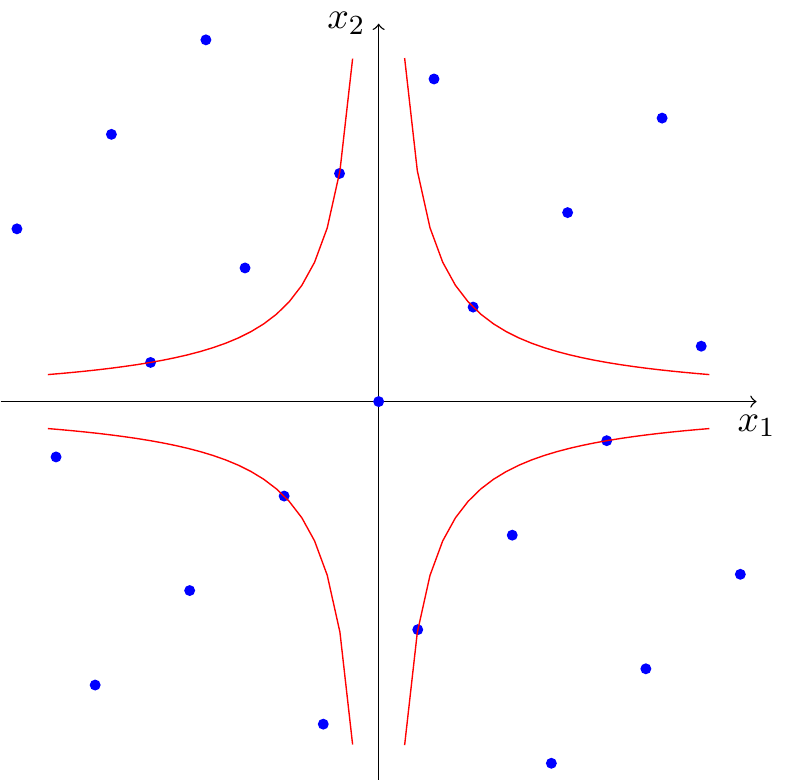}
	\caption{Admissible lattice and hyperbolic cross.} \label{fig_lattice_hc}
\end{figure}

\begin{defi}[Lattice]
A (full-rank) lattice $\Gamma\subset\R$ is a subgroup of $\R$ which is isomorphic to $\Z$ and spans the real vector
space $\R$. A set $\{a_1,...,a_d\} \subset \Gamma$ such that $\mbox{\rm span}_{\zz} \{a_1,...,a_d\} = \Gamma$ is called generating set of $\Gamma$. 
The matrix $A = (a_1|\cdots|a_d) \in \mbox{\rm GL}_d$
is called a generating matrix for $\Gamma$, i.e., we can write
$$
   \Gamma:=\{Ak:k\in \Z\}\,.
$$
\end{defi}
\noindent Let us further introduce the dual lattice. 
\begin{defi}[Dual lattice]
For a lattice $\Gamma \subset \R$ we define the dual lattice $\Gamma^{\bot}$ as
$$\Gamma^\bot =\{ x\in\R : (x,y)\in\zz \text{ for all } y\in\Gamma\}\,.$$
If $A$ is a generating matrix for $\Gamma$ then $A^{-\top}$ is a generating matrix for $\Gamma^{\bot}$.
\end{defi}
\noindent Crucial for the performance of the Frolov cubature formula \eqref{Fr} will be the notion of ``admissibility'' which is settled in the following definition.

\begin{defi}[Admissible lattice]
A lattice $\Gamma$ is called admissible if
$$\mathrm{Nm} (\Gamma) := \inf_{\gamma\in\Gamma\setminus\{0\}} \Big|\prod_{i=1}^d \gamma_i\Big| > 0$$
holds true.
\end{defi}
\noindent Figure \ref{fig_lattice_hc} illustrates this property. In fact, lattice points different from $0$ lie outside of a hyperbolic
cross with ``radius'' $\mathrm{Nm}(\Gamma)$.

The following lemma is essentially \cite[Lem.\ 3.1/2]{Sk94}. In the special case of a Vandermonde generator \eqref{vanderm} we refer to \cite[Lem.\ 2.1]{UU15}.

\begin{lem}\label{dual:adm} If a lattice $\Gamma \subset \R$ is admissible then $\Gamma^\bot \subset \R$ is also admissible.
 \end{lem}
 
\noindent There is a generic way to construct an admissible lattice described in Temlyakov \cite[IV.4]{Tem93}. 
For a polynomial $P(x)\in\zz [x]$ of order $d$ which is irreducible over $\qe$ and has 
$d$ different real roots $\xi_1, \xi_2,\ldots,\xi_d$ one can define the Vandermonde matrix 
$A=(a_{kl})_{k,l=1}^d=(\xi_k^{l-1})_{k,l=1}^d$, see \eqref{vanderm} above, which generates an admissible lattice $\Gamma$ with
$\mathrm{Nm}(\Gamma) = 1$. We will call such a generating matrix {\em Frolov matrix} since this construction has been already used by Frolov \cite{Fr76}.
Frolov originally used the construction to define the matrix $B$ which generates the dual lattice, and then $A=B^{-\top}$ 
was chosen as the lattice generator in the Frolov cubature formula. The reason is that convergence properties of the method require admissibility of the dual lattice. However, 
in \cite[Lem.\ 3.1]{Sk94} Skriganov has shown (see Lemma \ref{dual:adm} above) that if $B$ generates an admissible lattice, so does $A$, which means that both $B$ and $A$ are valid matrices for the Frolov cubature formula. A Frolov matrix with a small determinant is desirable since the Frolov cubature formula using this matrix will show (relatively) good preasymptotic behavior. The determinant of
$A=(a_{kl})_{k,l=1}^d=(\xi_k^{l-1})_{k,l=1}^d$ is given by
$$\det(A) = \prod_{k\neq l}(\xi_k-\xi_l)\,.$$
Therefore we need polynomials $P$ which additionally have accumulated roots. To find such polynomials is a challenging task, however, for certain dimensions there are results 
available which will be given in Section 3. 

Let us now consider different representations of a given lattice $\Gamma_A:=A(\Z)$ generated by $A \in \mbox{GL}_d$. This representation is not unique, because any linear automorphism $S$ on $\Z$ yields $S(\Z) = \Z$ and consequently $AS(\Z) = A(\Z)$. 
This gives rise to the question which lattice representation is favorable from the numerical point of view, cf. Figure \ref{fig_equiv_lattice}. In the special case of orthogonal lattices, the orthogonal representation stands out obviously.

\begin{figure}[t]
\centering
\includegraphics[width=0.48\linewidth]{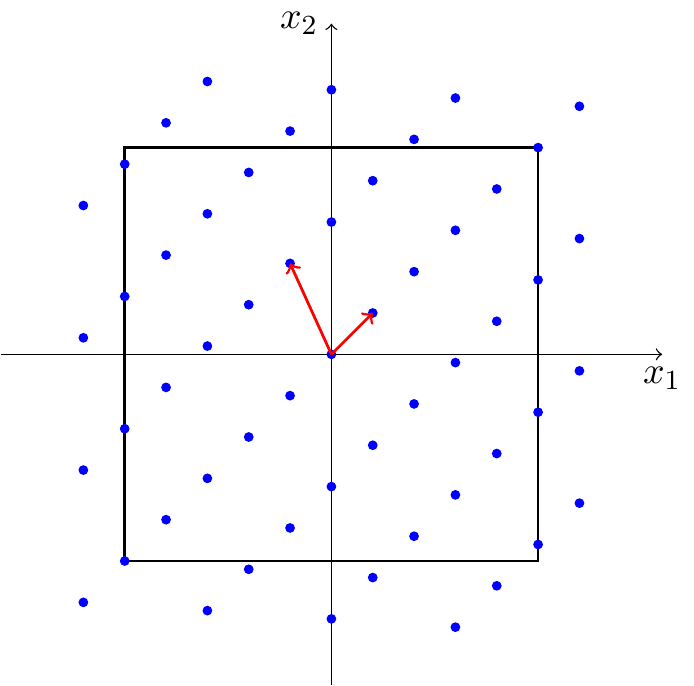}
\includegraphics[width=0.48\linewidth]{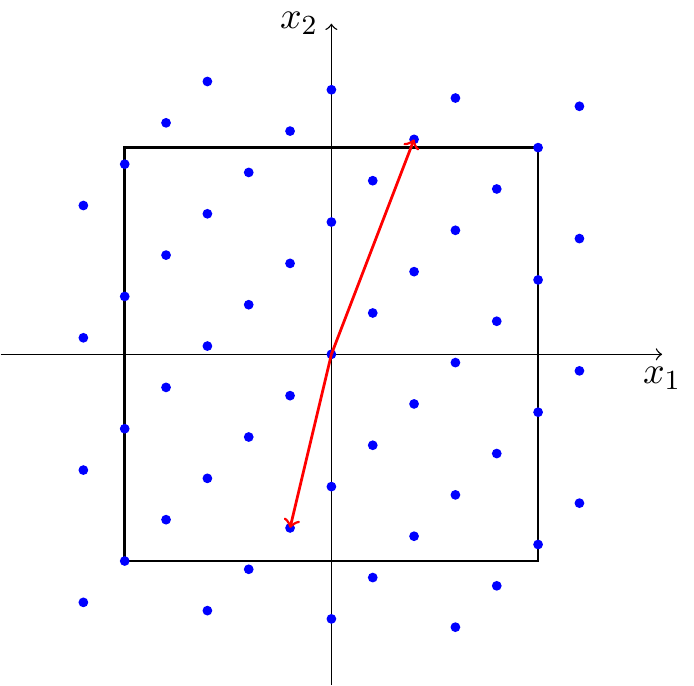}
\caption{Equivalent lattice representations within the unit cube $\Omega=\left[-1/2, 1/2\right]^2$.} \label{fig_equiv_lattice}
\end{figure}

\begin{defi}[Orthogonal lattice]
A lattice $\Gamma$ is called orthogonal if there exists a generating matrix $A\in \mbox{\rm GL}_d(\re)$ which has orthogonal column
vectors.
\end{defi}

\noindent In general, the computation of an orthogonal representation for an orthogonal lattice is 
performed by a discrete variant of the Gram-Schmidt method, 
e.g. the Lenstra-Lenstra-Lov\'{a}sz--lattice basis reduction algorithm (LLL), 
see \cite{Lenstra1982} or its modifications. 
However, as it turns out, in the case of Chebyshev-lattices an orthogonal basis can be 
determined a priori without any additional computational effort as we will show in the following Section. 

\section{Orthogonality of Chebyshev lattices} \label{sec_main}
Let $d\in\mathbb{N}$ and consider the Vandermonde matrix $T=(\xi_k^{l-1})_{k,l=1}^d$,
where 
$$\xi_k = 2\cos\left(\pi\frac{2k-1}{2d}\right)\quad,\quad k=1,\ldots,d\,,$$
represent the roots of $Q_d(x) = 2T_d(\frac{x}{2})\in\mathbb{Z}[x]$ and $T_d$ denotes the $d$-th Chebyshev polynomial. The lattice
$\Gamma_T = T(\Z)$ will be called {\em Chebyshev lattice}, and it is admissible if and only if $d=2^m$, see \cite{Tem93}, 
in which case we will call it {\em Chebyshev-Frolov lattice}. In fact, it is easy to show that for $d\neq 2^m$ the polynomial $Q_d(x)$ has a divisor which itself is a 
scaled Chebyshev polynomial $Q_{d'}(x)$ of lower order $d'=2^{m'}$ for some $m'\in\mathbb{N}$.

Our main result reads as follows. 
\begin{satz}
The Chebyshev lattice $\Gamma_T = T(\mathbb{Z}^d)$ is an orthogonal lattice.
\end{satz}
To show this, we will derive a lattice representation matrix $\tilde{T} = TS$, $S\in \mbox{SL}_d(\mathbb{Z})$ and show that it has
orthogonal column vectors.
\begin{lem}
For $\omega\in\mathbb{R}$ and $l\in \N$ define $\eta_l=2\cos(l\omega\pi)$. Then 
$$\eta_1^l-\eta_l\in\mathbb{Z}[\eta_1, \dots, \eta_{l-1}]\quad,\quad l\in \N.$$
More precisely, there exist integers $m^{(l)}_j\in\mathbb{Z}$ independent of $\omega$ such that for any $l\in \N$
$$\eta_1^l-\eta_l = m_0 + \sum_{j=1}^{l-1}m^{(l)}_j\eta_j\,.$$
\end{lem}
\begin{proof}

The proof is a straightforward calculation using Euler's formula by putting 

\begin{eqnarray*}
\eta_1^l-\eta_l&=&\left(e^{\omega\pi \mathrm{i}}+e^{-\omega\pi \mathrm{i}}\right)^l-\left(e^{\omega\pi \mathrm{i} l}+e^{-\omega\pi \mathrm{i} l}\right)\\
&=&\displaystyle\sum_{j=0}^l\binom{l}{j}e^{\omega\pi \mathrm{i}(l-2j)}-\left(e^{\omega\pi \mathrm{i} l}+e^{-\omega\pi \mathrm{i} l}\right)\\
&=&\displaystyle\sum_{j=1}^{l-1}\binom{l}{j}e^{\omega\pi \mathrm{i} (l-2j)}\\
&=&\begin{cases}
	\displaystyle\sum_{j=1}^{\frac{l-1}{2}}\binom{l}{j}\left(e^{\omega\pi \mathrm{i} (l-2j)}+e^{-\omega\pi \mathrm{i} (l-2j)}\right)&:
    l \text{ odd}\,,\\
    \displaystyle\sum_{j=1}^{\lfloor\frac{l-1}{2}\rfloor}\binom{l}{j}\left(e^{\omega\pi \mathrm{i} (l-2j)}+e^{-\omega\pi \mathrm{i} (l-2j)}\right)+\binom{l}{\frac{l}{2}}e^{\omega\pi \mathrm{i} (l-l)}&:
    l \text{ even}
\end{cases}\\
&=&\begin{cases}
	\displaystyle\sum_{j=1}^{\frac{l-1}{2}}\binom{l}{j}2\cos(\omega\pi(l-2j)) & : 
    l \text{ odd}\,,\\
    \displaystyle\sum_{j=1}^{\lfloor\frac{l-1}{2}\rfloor}\binom{l}{j}2\cos(\omega\pi(l-2j))+\binom{l}{\frac{l}{2}}&: 
    l \text{ even}\,.
\end{cases}
\end{eqnarray*}
The values $m_j^{(l)}$ can be obtained from this representation.
\end{proof}
This lemma leads to our desired lattice representation, since multiplying with a matrix $S\in \mbox{SL}_d(\mathbb{Z})$
from the right is a composition of column operations.

\begin{cor} \label{cor_matrix}
The matrix $\tilde{T} = TS$, where $S\in\ \mbox{SL}_d(\mathbb{Z})$ is a suitable column operation matrix, given by
$$\tilde{T}_{kl} = 
\begin{cases}
	1&:
    l=1\,,\\
	2\cos\left(\pi (l-1)\frac{2k-1}{2d}\right)&:
    l=2\dots d
\end{cases} $$
generates the lattice $\Gamma_T = T(\mathbb{Z}^d)$.
\end{cor}
\begin{proof}
The case $d=2$ is trivial, so assume $d>2$. For $i=3\dots d$ we define $S^{(l)}\in \mbox{SL}_d(\zz)$ to be a column operation matrix changing the $l$-th column:
\begin{equation}\label{Sl}
    S^{(l)} = \left(\begin{array}{cccccc}
                1 &  &  & -m_0^{(l)} &  & \\
                  & \ddots &  & \vdots & & \\
                  &  & \ddots & -m_{l-2}^{(l)} & & \\
                  &  &  & 1& & \\
									&  &  &  & \ddots & \\
									&  &  &  &  & 1
              \end{array}\right)\,
\end{equation}
Then the product matrix $S = S^{(3)}\cdots S^{(d)}$ consecutively transforms the entries of $T$ which have the
form $\xi_k^{l-1}=2\cos(\pi\frac{2k-1}{2d})^{l-1}$ according to Lemma 3.2.
\end{proof}
We remark that this formula is applicable in general to any Vandermonde lattice with generating factors ranging from
$-2$ to $2$. Furthermore, $\tilde{T}$ has better stability properties than $T$. The following lemma will complete the proof of
Theorem \ref{main}.
\begin{lem} The matrix $\tilde{T}$ is orthogonal. Moreover, it holds $\tilde{T}^{\top}\tilde{T} = \mathrm{diag}(d,2d,...,2d)$\,.
\end{lem}
\begin{proof}
For $l=2\dots d$ we have
\begin{eqnarray*}
((\tilde{T})^{\top}\tilde{T})_{1l} &=& \displaystyle\sum_{k=1}^d 2\cos\left(\pi (l-1)\frac{2k-1}{2d}\right)\\
&=& \displaystyle\sum_{k=1}^d \left(e^{\mathrm{i} \pi (l-1)\frac{2k-1}{2d}}+e^{-\mathrm{i} \pi (l-1)\frac{2k-1}{2d}}\right)\\
&=& \displaystyle\sum_{k=1}^{2d}e^{\mathrm{i} \pi (l-1)\frac{2k-1}{2d}}\,. \\
\end{eqnarray*}
We continue observing 
\begin{equation*}
	\sum_{k=1}^{2d}e^{2 \pi \mathrm{i} (l-1)\frac{2k-1}{4d}} = \left(\displaystyle\sum_{k=1}^{2d}e^{2 \pi \mathrm{i} (l-1)\frac{k}{2d}}\right)e^{\frac{{-2 \pi \mathrm{i} (l-1)}}{4d}} = \frac{1-e^{2\pi \mathrm{i}(l-1)}}{1-e^{2 \pi \mathrm{i} \frac{(l-1)}{2d}}}e^{\frac{{-2 \pi \mathrm{i} (l-1)}}{4d}} = 0 \,.
\end{equation*}
Let us now consider $l=2\dots d$ and $j=2\dots d$. We find
\begin{eqnarray*}
((\tilde{T})^{\top}\tilde{T})_{jl} &=& \displaystyle\sum_{k=1}^d 2\cos\left(\pi (j-1)\frac{2k-1}{2d}\right)2\cos\left(\pi
(l-1)\frac{2k-1}{2d}\right)\\
&=&\displaystyle\sum_{k=1}^d\left(e^{ \pi \mathrm{i} (j-1)\frac{2k-1}{2d}}+e^{- \pi \mathrm{i} (j-1)\frac{2k-1}{2d}}\right) \left(e^{\pi \mathrm{i}
(l-1)\frac{2k-1}{2d}}+e^{-\pi \mathrm{i} (l-1)\frac{2k-1}{2d}}\right)\\
&=&\displaystyle\sum_{k=1}^d e^{\pi \mathrm{i} (j+l-2)\frac{2k-1}{2d}}+e^{\pi \mathrm{i} (j-l)\frac{2k-1}{2d}}+e^{\pi \mathrm{i}
(l-j)\frac{2k-1}{2d}}+e^{-\pi \mathrm{i} (j+l-2)\frac{2k-1}{2d}}\\
&=&\displaystyle\sum_{k=1}^{2d}e^{2\pi \mathrm{i} (j+l-2)\frac{2k-1}{4d}} + \sum_{k=1}^{2d}e^{2\pi \mathrm{i} (j-l)\frac{2k-1}{4d}}
=
\begin{cases}
	2d&: j=l\\
	0 &:\text{otherwise}\,.
\end{cases}
\end{eqnarray*}
\end{proof}

\section{The Frolov cubature formula}
We return to the Frolov cubature formula \eqref{Fr} mentioned in the introduction, see \cite{Fr76, Sk94, Tem93,
Tem03, UU15}, to estimate integrals of the form
$$
	I(f):=\int_{\Omega} f(x)\,\mathrm{d}x\,,
$$
where $\Omega \subset \R$ is a compact set. The matrix $T \in \mathbb{R}^{d \times d}$ is chosen such that $T(\Z)$ is an admissible lattice, for
instance the Chebyshev-Frolov matrix from above. 
For a given scaling parameter $n \in \N$ we define the matrix
\begin{equation}
   \mathcal{T}_n = (n \det(T))^{-\frac{1}{d}}T ,
\end{equation}
which satisfies $\det(\mathcal{T}_n )=1/n$. Defining $\Gamma_n = \mathcal{T}_n(\mathbb{Z}^d)$, the integration nodes are chosen as the
elements of the lattice $\Gamma_n$ belonging to $\Omega$, i.e. $N(n) := | \Gamma_n \cap \Omega |$. Note, that the cubature 
weights of the Frolov method are chosen to be uniformly $1/n$.
But, despite the uniformity of the weights, the Frolov cubature formula does not represent a Quasi--Monte Carlo
method since in general $N(n) \neq n$, i.e., the weights do not sum up to one. However,
 we have that $\lim_{n \to \infty } \tfrac{|\Gamma_n\cap\Omega|}{n} = \mathrm{vol}(\Omega)$. 

The formula \eqref{Fr} 
performs asymptotically optimal for a broad variety of function spaces with dominating mixed smoothness, see \cite{UU15}. 
To this end, we define the Besov spaces of dominating mixed smoothness as follows. 

\begin{defi}[Besov space of mixed smoothness] \label{def:besov}
Let $0 < p,\theta\le\infty$, $r > \max\{1/p-1,0\}$, and 
$(\varphi_m)_{m\in\N_0^d}$ be a tensorized decomposition of unity in the sense of \cite[Rem.\ 3.3]{DTU16}. 
The \emph{Besov space of dominating mixed smoothness} 
$\Bspt=\Bspt(\R)$ is the set of all $f\in L_1(\R)$ 
such that 
\[
\|f\|_{\Bspt} \,:=\, 
\Big(\sum_{m\in\N_0^d} 2^{ r|m|_1 \theta}\, 
	\|\cf^{-1}[\varphi_m \cf f]\|_p^\theta\Big)^{1/\theta} \,<\,\infty
\]
with the usual modification for $\theta=\infty$. 
\end{defi}
\noindent In the special case $p=\theta=2$ we put $H^r_{\text{mix}}(\R) := \Bspt$ which denotes the 
\emph{Sobolev spaces of dominating mixed smoothness} $r$. Let us restrict to the case $\Omega :=
[-1/2,1/2]^d$ in the sequel and define a subspace of $\Bspt$, namely the space of $\Bo$ of functions which are
supported in the unit cube $[-1/2,1/2]^d$, 
i.e.~we consider 
\begin{equation}\label{Ao}
\Bo \,:=\, \bigl\{ f\in \Bspt(\R)\colon \supp(f)\subset[-1/2,1/2]^d\bigr\}\,.
\end{equation}
In \cite{Du97, UU15} it has been shown in case $1\leq p,\theta\leq \infty$ and $r>1/p$ that 
\begin{equation}\label{rate}
     \sup\limits_{f\in \Bo, \|f\|_{\Bspt}\leq 1}|I(f)-\Phi(T,n;f)| \asymp n^{-r}(\log n)^{(d-1)(1-1/\theta)}\,,
\end{equation}
where the constant behind $\asymp$ depends on $d$ and the choice of $T$. Note, that the rate in \eqref{rate} is independent of 
the integrability parameter $p$. Taking into account that the number $N(n)$ of cubature nodes satisfies
\begin{equation}\label{numnodes}
      N(n):=n+O(\log^{d-1} n)\,,
\end{equation}
see \cite[(0.1)]{Sk94}, the rate of convergence \eqref{rate} is optimal among all cubature formulas with $N$ arbitrary nodes
and weights. 

\section{Enumerating the Chebyshev-Frolov nodes}

%
%
%
%
%
%

In order to generate the Frolov cubature nodes belonging to $\Omega:=[-1/2,1/2]^d$ explicitly one needs an efficient way to enumerate all points from
$\Gamma_n \cap \Omega$, which already in moderate dimensions is a difficult task.
In fact, we need to determine
\begin{equation}\label{Xn}
	X_n := \mathcal{T}_n (\mathbb{Z}^d) \cap \left[-\frac{1}{2}, \frac{1}{2} \right]^d 
\end{equation}
as efficient as possible. This is equivalent to finding the pre-image of $[-1/2,1/2]^d$ under the linear map $\mathcal{T}_n$ intersected with $\Z$, i.e.
\begin{equation}
	Y_n := \left( \mathcal{T}_n^{-1} \left[-\frac{1}{2}, \frac{1}{2} \right]^d \right) \cap \mathbb{Z}^d ,
\end{equation}
since $k \in Y_n$ if and only if $\mathcal{T}_n k \in X_n$. Now it is a natural approach to use a finite set $K_n \subset \mathbb{Z}^d$ that covers $Y_n$, i.e. $Y_n \subset K_n$, and allows for an efficient enumeration 
on a computer. Then, one can check for each vector $k \in K_n$ wether $\mathcal{T}_n k \in X_n$. 

\begin{figure}[t]
\centering
\includegraphics[width=0.48\linewidth]{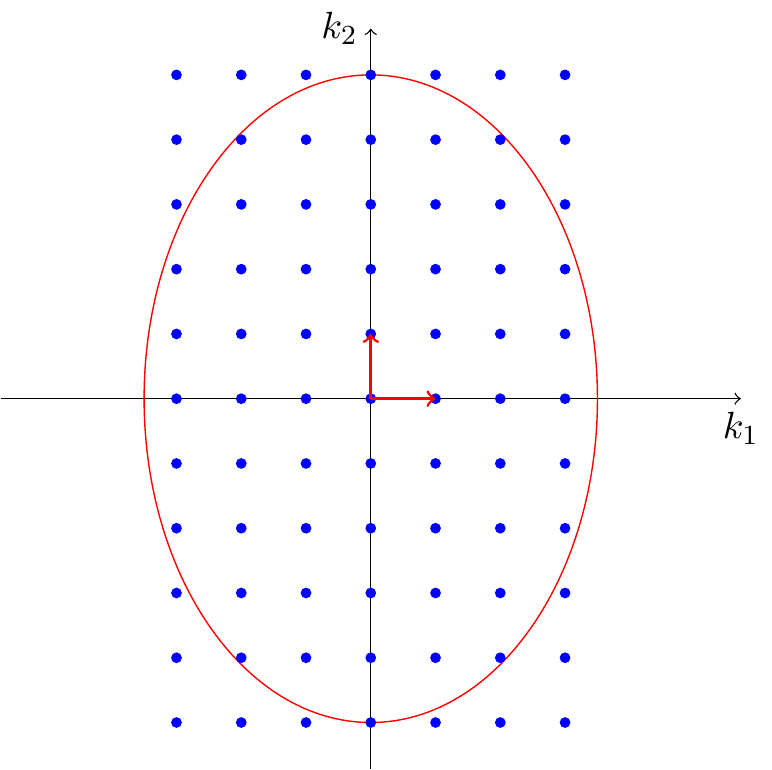}
\includegraphics[width=0.48\linewidth]{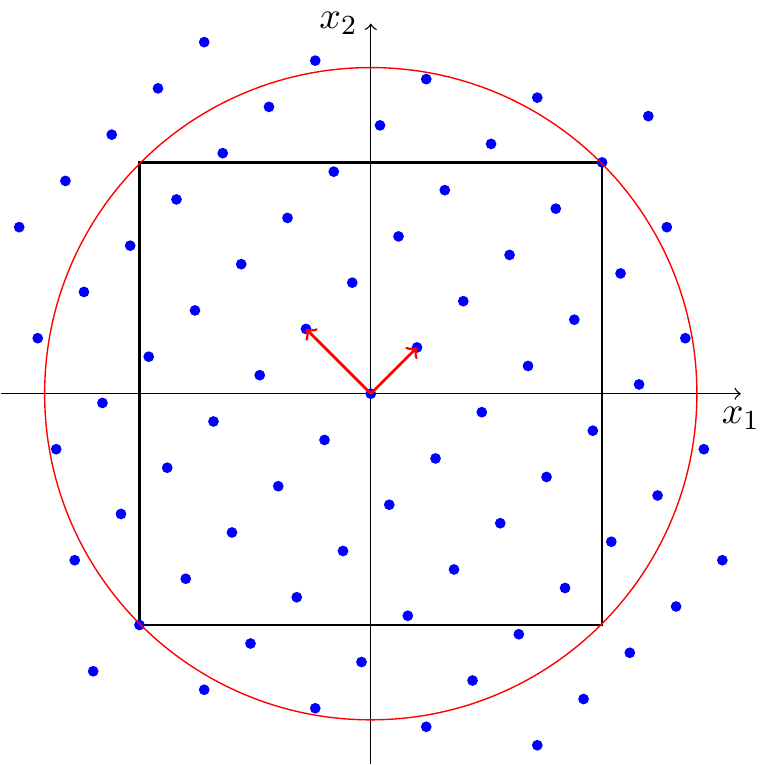}
\caption{The ellipsoid (left) that is the pre-image under $\mathcal{T}_n$ of the bounding ball of $\Omega=[-1/2,1/2]^2$ (right).}
\label{fig_integerside}

\end{figure}

However, there remains the problem of determining suitable covering sets $K_n$. To this end, we note that an efficient enumeration is possible at least for all integer vectors within $\ell_p$-ellipsoids that are axis-aligned, i.e.
\begin{equation}
	E_{p,d}(R;\mu_1,\ldots,\mu_d) := \Big\{x \in \mathbb{R}^d: \sum_{j=1}^d \Big| \frac{x_j}{\mu_j} \Big|^p \leq R^p \Big\} , 
\end{equation}
where $R\mu_1,\ldots, R\mu_d > 0$ denote the lengthes of the semi-axes. An efficient enumeration of all integer vectors belonging to such a set is possible 
due to the recursive representation of its discrete counterpart 
$\mathcal{E}_{p,d}(R;\mu_1,\ldots,\mu_d) := E_{p,d}(R;\mu_1,\ldots,\mu_d) \cap \Z$, which reads
\begin{equation}\label{capZ}
	\mathcal{E}_{p,d}(R;\mu_1,\ldots,\mu_d) = \bigcup_{|k_d| \leq R \, \mu_d} \mathcal{E}_{p,d-1}\left( \left[R^p- 
	\left( \frac{|k_d|}{\mu_d} \right)^p\right]^{\frac{1}{p}};\mu_1,\ldots,\mu_{d-1} \right) \times \{k_d\}
\end{equation}
and can easily be implemented as a $d$-fold nested for-loop. In addition, the cardinality of \eqref{capZ}, i.e., 
the number of integer vectors belonging to $E_{p,d}(R;\mu_1,\ldots,\mu_d)$ can be estimated 
following the approach in \cite[Sect.\ 3]{KuMaUl15}. To this end, we have the following result, 
which relates the number of integer vectors within a general $p$-ellipsoid to its volume. 
Note at this point the relation $E_{p,d}(R;1,\ldots,1) = R\cdot B_p^d$, where $B_p^d$ denotes the unit-ball with respect to the $\ell_p$-(quasi)-norm.
 
\begin{prop}\label{enum_ellipse} Let $\mu = (\mu_1,...,\mu_d)>0$ and $0<p<\infty$. \\
	{\em (i)} For the volume of the ``unit'' ellipsoid $E_{p,d}(1;\mu_1,\ldots,\mu_d)$ it holds
	$$
	      \mathrm{vol}(E_{p,d}(1;\mu_1,\ldots,\mu_d)) = \mathrm{vol}(B_p^d)\prod_{j=1}^d \mu_j = 2^d\frac{\Gamma(1+1/p)^d}{\Gamma(d/p+1)} \prod_{j=1}^d \mu_j\,.
	$$
	{\em (ii)} If $R>r(\mu,p)$ then the number of integer points in $E_{p,d}(R;\mu_1,\ldots,\mu_d)$ is bounded from above and below by
	\begin{equation*}
		 \left(R^{\varrho} - r(\mu, p)^{\varrho} \right)^{d/\varrho} \mathrm{vol}(E_{p,d}(1;\mu_1,\ldots,\mu_d)) \leq |\mathcal{E}_{p,d}(R;\mu)| 
		 \leq \left(R^{\varrho} + r(\mu, p)^{\varrho} \right)^{d/\varrho} \mathrm{vol}(E_{p,d}(1;\mu_1,\ldots,\mu_d))\,,
	\end{equation*}
	where $r(\mu, p): = 1/2 \, \left(\sum_{j=1}^d |\mu_j|^{-p} \right)^{1/p}$ and $\varrho:=\min\{p,1\}$.\\
{\rm (iii)} It holds
$$
      \lim\limits_{R \to \infty} \frac{|\mathcal{E}_{p,d}(R;\mu)|}{R^d} = \mathrm{vol}(E_{p,d}(1;\mu_1,\ldots,\mu_d))\,.
$$
\end{prop}
	
\begin{proof} The formula in (i) is obtained by change of variable and the well-known formula for the volume of standard $\ell_p$-balls in $\re^d$. In fact, we have 
$$
      \int\limits_{\{x \in \re^d~:~\sum_i |x_i/\mu_i|^p \leq 1\}} 1\,\mathrm{d}x = \Big(\prod\limits_{j=1}^d \mu_j\Big) \int\limits_{|y|_p\leq 1} 1\,dy = 
      2^d\frac{\Gamma(1+1/p)^d}{\Gamma(d/p+1)}\prod\limits_{j=1}^d \mu_j\,.
$$
The limit statement in (iii) is a direct consequence of (ii). 

It remains to prove (ii). Here we use the arguments in \cite[Sect.\ 3]{KuMaUl15} and define a (quasi-)norm on $\re^d$ via
$$
    \|x\|:=\Big(\sum\limits_{i=1}^d |x_i/\mu_i|^p\Big)^{1/p}\quad,\quad x\in \re^d\,.
$$
Note, that the classical triangle inequality is replaced by the $\varrho$-triangle inequality, where $\varrho:=\min\{1,p\}$, i.e., 
$\|x+y\|^{\varrho} \leq \|x\|^{\varrho} + \|y\|^{\varrho}$ for all $x,y\in \re^d$\,. We denote with 
$B_{\|\cdot\|}$ the (closed) unit ball of $(\re^d,\|\cdot\|)$. By putting 
$Q_k:=k+[-1/2,1/2]^d$ for $k\in \zz^d$ we observe according to \cite[Sect.\ 3]{KuMaUl15} as a consequence of the $\varrho$-triangle inequality 
$$
      \ell(R,p,d)B_{\|\cdot\|}\subset \bigcup\limits_{\|k\|\leq R} Q_k \subset L(R,p,d)B_{\|\cdot\|}\,,
$$
where $\ell(R,p,d):=(R^{\varrho}-r(\mu,p)^{\varrho})^{1/\varrho}$ and $L(R,p,d):=(R^{\varrho}+r(\mu,p)^{\varrho})^{1/\varrho}$ with 
$r(\mu,p)=\|\sum_{i=1}^d e_i\|/2$\,. Taking volumes on both sides yields (ii). 
\end{proof}
\noindent Now we are in the position to exploit the orthogonality of the Chebyshev-Frolov lattice by choosing a proper bounding ellipsoid with respect to the Euclidian norm, i.e., $p=2$. 
To this end, we write $\mathcal{T}_n = (n \det(D))^{-\frac{1}{d}} Q D$ as the scaled product of an orthogonal matrix with unit determinant 
$Q \in \mathrm{SO}_d$ and a diagonal matrix $D = \mbox{diag}(\lambda_1,...,\lambda_d)$ with entries 
\begin{equation}\label{lambdas}
   \lambda_1 = \sqrt{d}\quad\mbox{and}\quad \lambda_2 = \ldots = \lambda_d = \sqrt{2 d}\,.
\end{equation}
We note that it holds $[-1/2, 1/2]^d \subset E_{2,d}(\sqrt{d} / 2;1,\ldots,1)= (\sqrt{d}/2)B_2^d$, the isotropic ball in $\R$ of radius $\sqrt{d} / 2$. Therefore we can compute
\begin{align*}
	\mathcal{T}_n^{-1} [-1/2, 1/2]^d & \subset \mathcal{T}_n^{-1} E_{2,d}(\sqrt{d} / 2;1,\ldots,1)\\
				&	=  (n \det(D))^{\frac{1}{d}} D^{-1} Q^T  E_{2,d}(\sqrt{d} / 2;1,\ldots,1) \\
		    & = (n \det(D))^{\frac{1}{d}} D^{-1} E_{2,d}(\sqrt{d} / 2;1,\ldots,1) \\
		    & = E_{2,d}\left(R_n;\lambda_1^{-1},\ldots,\lambda_d^{-1} \right) ,
\end{align*}
where $R_n = \frac{\sqrt{d}}{2} (n \det(D))^{\frac{1}{d}} =\frac{\sqrt{d}}{2} n^{1/d} \prod_{j=1}^d \lambda_j^{1/d}$. The discrete $\ell_2$-ellipsoid 
\begin{equation} \label{eqn_coveringset}
    K_n := \mathcal{E}_{2,d}(R_n; \lambda_1^{-1},\ldots,\lambda_d^{-1}) := E_{2,d}\left(R_n;\lambda_1^{-1},\ldots,\lambda_d^{-1} \right) \cap \Z
\end{equation}
is our desired, easily accessible finite set that covers the pre-image of $X_n$. 
In order to determine the complexity of our enumeration algorithm, we have to bound the cardinality $|K_n|$ of $K_n$. As a special case of Proposition \ref{enum_ellipse} 
we obtain the following result on this cardinality.

\begin{satz} Let $K_n$, $X_n$ be given by \eqref{eqn_coveringset}, \eqref{lambdas} and \eqref{Xn}. \\
{\em (i)} If $n>2^{3d/2}$ then the cardinality $|K_n|$ is bounded from below and above by
\begin{equation}\label{K_n_bound}
	n \Big( 1  -  \frac{2^{3/2}}{n^{1/d}} \Big)^d \frac{(d\pi)^{d/2}}{2^d\Gamma(d/2+1)} \quad \leq \quad |K_n| \quad \leq \quad n \, 
	\Big(1  +  \frac{2^{3/2}}{n^{1/d}} \Big)^d \frac{(d\pi)^{d/2}}{2^d\Gamma(d/2+1)}\,.
\end{equation}
{\em (ii)} As a consequence, we obtain the limit statements
\begin{equation}\label{f3}
      \lim\limits_{n\to \infty} |K_n|/n = \lim\limits_{n\to\infty} |K_n|/|X_n| = \mathrm{vol}((\sqrt{d}/2)B_2^d)  \leq \Big(\frac{\pi e}{2}\Big)^{d/2}\,.
\end{equation}
\end{satz}
\begin{proof}
We apply Proposition \ref{enum_ellipse} with $p=2$, $\mu_1=d^{-1/2}$, $\mu_2= ... = \mu_d = (2d)^{-1/2}$ and $R_n = (\sqrt{d}/2)n^{1/d}\prod_{j=1}^d \mu_j^{-1/d}$
for $n\in \N$. Due to $\Gamma(3/2) = \sqrt{\pi}/2$ we obtain from Proposition \ref{enum_ellipse}
$$
      \frac{\pi^{d/2}}{\Gamma(d/2+1)}\prod\limits_{j=1}^d \mu_j = \lim\limits_{n\to \infty} \frac{|\mathcal{E}_{2,d}(R_n;\mu)|}{R_n^d}
      = \Big(\prod\limits_{j=1}^d \mu_j\Big)(\sqrt{d}/2)^{-d}\lim\limits_{n\to\infty} |K_n|/n\,,
$$
which immediately implies the second identity in \eqref{f3}. Due to $|X_n| = N(n) = n + O((\log n)^{d-1})$, see \eqref{numnodes} above, we obtain the first identity. 
The inequality is a consequence of $\Gamma(1+x) \geq (x/e)^x$. 

It remains to prove \eqref{K_n_bound}. By Proposition \ref{enum_ellipse}, (ii), we have ($\varrho=1$)
\begin{equation}\label{f4}
      (R_n-r(\mu,2))^d\frac{\pi^{d/2}}{\Gamma(d/2+1)} \prod_{j=1}^d \mu_j \leq |K_n|  \leq (R_n+r(\mu,2))^d\frac{\pi^{d/2}}{\Gamma(d/2+1)} \prod_{j=1}^d \mu_j\,.
\end{equation}
By the definition of the $R_n$ we have $R_n\prod_{j=1}^d \mu_j^{1/d} = n^{1/d}\sqrt{d}/2$\,. Moreover, the special choice of the $\mu_j$'s gives
$$
   r(\mu,2)\prod_{j=1}^d \mu_j^{1/d}= 2^{-\frac{d-1}{2d}}\cdot \sqrt{2d^2-d}\frac{1}{\sqrt{d}} \leq \sqrt{2d}\,.
$$
Plugging this into \eqref{f4} yields \eqref{K_n_bound}.\end{proof}

One can see, that the cardinality $|K_n|$ scales linear in $n$, where the factor depends exponentially on the dimension $d$.
The true number of discrete lattice points in $K_n$ that have to be ``seen'' is given in Table \ref{tab_complexity2d} for dimensions $d=2,4,8,16$. 
The relative overhead $|K_n||X_n|^{-1}$ converges to a constant smaller than $2.07^d$ for $n$ tending to infinity, which outlines the complexity of the enumeration algorithm with respect to $N$, 
where $N = N(n) = |X_n|$. 

\section*{Acknowledgment}
The authors acknowledge the fruitful discussions with D. Bazarkhanov, A. Hinrichs, W. Sickel, V.N.
Temlyakov and M. Ullrich on the topic of this paper. Tino Ullrich gratefully acknowledges support by the German Research
Foundation (DFG) and the Emmy-Noether programme, Ul-403/1-1. Jens Oettershagen was supported by the DFG via project GR-1144/21-1 and the CRC $1060$.

\newpage

\section*{Appendix}

\begin{table}[H]
\centering
\begin{tabular}{|l|l|l|l|}
\hline
\multicolumn{4}{ |c| }{Dimension $d=2$}\\
\hline
scaling factor $n$ & cubature points in $X_n$ & ellipsoid points $|K_n|$ & relative overhead $|K_n||X_n|^{-1}$ \\
\hline
64&65&101&1.55\\
\hline
256&257&409&1.59\\
\hline
1024&1027&1599&1.56\\
\hline
4096&4095&6427&1.57\\
\hline
16384&16383&25735&1.57\\
\hline
65536&65539&102951&1.57\\
\hline
262144&262145&411813&1.57\\
\hline
1048576&1048579&1647103&1.57\\
\hline
\end{tabular}

\begin{tabular}{|l|l|l|l|}
\hline
\multicolumn{4}{ |c| }{Dimension $d=4$}\\
\hline
scaling factor $n$ & cubature points in $X_n$ & ellipsoid points $|K_n|$ & relative overhead $|K_n||X_n|^{-1}$\\
\hline
64&71&347&4.89\\
\hline
256&261&1205&4.62\\
\hline
1024&1025&5061&4.94\\
\hline
4096&4099&20287&4.95\\
\hline
16384&16385&81105&4.95\\
\hline
65536&65533&324241&4.95\\
\hline
262144&262143&1297123&4.95\\
\hline
1048576&1048609&5176701&4.95\\
\hline
\end{tabular}

\begin{tabular}{|l|l|l|l|}
\hline
\multicolumn{4}{ |c| }{Dimension $d=8$}\\
\hline
scaling factor $n$ & cubature points in $X_n$ & ellipsoid points $|K_n|$ & relative overhead $|K_n||X_n|^{-1}$\\
\hline
64&79&4459&56.44\\
\hline
256&271&15395&56.81\\
\hline
1024&1067&63299&59.32\\
\hline
4096&4113&267005&64.92\\
\hline
16384&16413&1077433&65.65\\
\hline
65536&65645&4231533&64.46\\
\hline
262144&262263&16729291&63.79\\
\hline
1048576&1048779&68078523&64.91\\
\hline
\end{tabular}

\begin{tabular}{|l|l|l|l|}
\hline
\multicolumn{4}{ |c| }{Dimension $d=16$}\\
\hline
scaling factor $n$ & cubature points in $X_n$ & ellipsoid points $|K_n|$ & relative overhead $|K_n||X_n|^{-1}$\\
\hline
64&423&751915&1777.58\\
\hline
256&967&4349507&4497.94\\
\hline
1024&2043&17758079&8692.16\\
\hline
4096&5835&58780787&10073.83\\
\hline
16384&18901&232153093&12282.58\\
\hline
65536&69353&969855677&13984.34\\
\hline
262144&267257&4086738257&15291.42\\
\hline
1048576&1054837&16642145301&15776.98\\
\hline
\end{tabular}

\caption{Cardinalities of the sets of Frolov-cubature points $X_n$, the bounding ellipsoids $K_n$ and the relative overhead.} \label{tab_complexity2d}
\end{table}

\end{document}